\documentclass[11pt,reqno]{amsart}
\usepackage{amssymb,latexsym}
\usepackage{graphicx}
\usepackage{mathtools}
\usepackage{color}
\usepackage[T1]{fontenc}
\usepackage{verbatim}
\usepackage{hyperref}
\usepackage{amsmath}
\usepackage{mathdots}
\usepackage{breqn}
\usepackage{url}    
\usepackage{breqn}
\newcommand{\bburl}[1]{\textcolor{blue}{\url{#1}}}

\calclayout

\makeatletter
\newcommand{\monthyear}[1]{%
  \def\@monthyear{\uppercase{#1}}}
\newcommand{\volnumber}[1]{%
  \def\@volnumber{\uppercase{#1}}}
\AtBeginDocument{%
\def\ps@plain{\ps@empty
  \def\@oddfoot{\@monthyear \hfil \thepage}%
  \def\@evenfoot{\thepage \hfil \@volnumber}}
\def\ps@firstpage{\ps@plain}
\def\ps@headings{\ps@empty
  \def\@evenhead{%
    \setTrue{runhead}%
    \def\thanks{\protect\thanks@warning}%
    \uppercase{\ }\hfil}%
  \def\@oddhead{%
    \setTrue{runhead}%
    \def\thanks{\protect\thanks@warning}%
    \hfill\uppercase{Generalizing Zeckendorf's Theorem}}%
  \let\@mkboth\markboth
  \def\@evenfoot{%
    \thepage \hfil \@volnumber}%
  \def\@oddfoot{%
    \@monthyear \hfil \thepage}%
  }%
\footskip=25pt
\pagestyle{headings}%
}
\makeatother

\theoremstyle{plain}
\numberwithin{equation}{section}
\newtheorem{thm}{Theorem}[section]
\newtheorem{theorem}[thm]{Theorem}
\newtheorem{lemma}[thm]{Lemma}
\newtheorem{conjecture}[thm]{Conjecture}
\newtheorem{corollary}[thm]{Corollary}
\newtheorem{example}[thm]{Example}
\newtheorem{definition}[thm]{Definition}

\theoremstyle{definition}
\newtheorem*{remark}{Remark}

\newcommand{\ignore}[1]{}

\newcommand{\pn}[1]{\left( #1 \right)}

\newcommand\be{\begin{eqnarray}}
\newcommand\ee{\end{eqnarray}}
\newcommand\bea{\begin{eqnarray}}
\newcommand\eea{\end{eqnarray}}
\newcommand\ben{\begin{enumerate}}
\newcommand\een{\end{enumerate}}

\begin{document}

\monthyear{August 2021}
\volnumber{}
\setcounter{page}{1}

\title{Generalizing Zeckendorf's Theorem to Homogeneous Linear Recurrences, I}

\author{Thomas C. Martinez, Steven J. Miller, Clayton Mizgerd, Chenyang Sun}


\address{\tiny{Department of Mathematics, Harvey Mudd College, Claremont, CA 91711}} \email{tmartinez@hmc.edu}

\address{\tiny{Department of Mathematics and Statistics, Williams College, Williamstown, MA 01267}} \email{sjm1@williams.edu, Steven.Miller.MC.96@aya.yale.edu}
\email{cmm12@williams.edu}
\email{cs19@williams.edu}


\date{\today}

\begin{abstract} Zeckendorf's theorem states that every positive integer can be written uniquely as the sum of non-consecutive shifted Fibonacci numbers $\{F_n\}$, where we take $F_1=1$ and $F_2=2$. This has been generalized for any Positive Linear Recurrence Sequence (PLRS), which informally is a sequence satisfying a homogeneous linear recurrence with a positive leading coefficient and non-negative integer coefficients. These decompositions are generalizations of base $B$ decompositions. In this and the followup paper, we provide two approaches to investigate linear recurrences with leading coefficient zero, followed by non-negative integer coefficients, with differences between indices relatively prime (abbreviated ZLRR). The first approach involves generalizing the definition of a legal decomposition for a PLRS found in Kolo\u{g}lu, Kopp, Miller and Wang. We prove that every positive integer $N$ has a legal decomposition for any ZLRR using the greedy algorithm. We also show that a specific family of ZLRRs loses uniqueness of decompositions. The second approach converts a ZLRR to a PLRR that has the same growth rate. We develop the Zeroing Algorithm, a powerful helper tool for analyzing the behavior of linear recurrence sequences. We use it to prove a very general result that guarantees the possibility of conversion between certain recurrences, and develop a method to quickly determine whether certain sequences diverge to $+\infty$ or $-\infty$, given any real initial values. This paper investigates the first approach.
\end{abstract}

\thanks{This work was supported by NSF Grants DMS1561945 and DMS1659037, as well as the Finnerty Fund. We thank the participants of the 2019 Williams SMALL REU and the referee for constructive comments.}

\maketitle

\tableofcontents



\section{Introduction and Definitions}\label{sec:intro}


\subsection{History and Past Results}\label{sec:history}

The Fibonacci numbers are one of the most well-known and well-studied mathematical objects, and have captured the attention of mathematicians since their conception. This paper focuses on a generalization of Zeckendorf's theorem, one of the many interesting properties of the Fibonacci numbers. Zeckendorf \cite{Ze} proved that every positive integer can be written \textbf{uniquely} as the sum of non-consecutive Fibonacci numbers (called the \textit{Zeckendorf Decomposition}), where the (shifted) Fibonacci numbers\footnote{If we use the standard initial conditions then 1 appears twice and uniqueness is lost.} are $F_1 = 1, F_2 = 2, F_3 = 3, F_4 = 5, \dots$. This result has been generalized to other types of recurrence sequences. We set some notation before describing these generalizations.

\begin{definition}[\cite{KKMW}, Definition 1.1, (1)]\label{def:plrrdefinition}
    We say a recurrence relation is a \textbf{Positive Linear Recurrence Relation (PLRR)} if there are non-negative integers $L, c_1, \dots, c_L$ such that
        \begin{equation}
            H_{n+1}\ =\ c_1\, H_n + \cdots + c_L\, H_{n+1-L},
        \end{equation}
        with $L, c_1$ and $c_L$ positive.
\end{definition}

\begin{definition}[\cite{KKMW}, Definition 1.1, (2)]\label{def:plrsdefinition}
    We say a sequence $\{H_n\}_{n=1}^{\infty}$ of positive integers arising from a PLRR is a \textbf{Positive Linear Recurrence Sequence (PLRS)} if $H_1=1$, and for $1 \leq n < L$ we have
        \begin{equation}
            H_{n+1}\ =\ c_1\,H_n + c_2\,H_{n-1} + \cdots + c_n \,H_1 + 1.
        \end{equation}
    We call a decomposition $N=\sum_{i=1}^m a_i H_{m+1-i}$ of a positive integer, and its associated sequence $\{a_i\}_{i=1}^m$, \textbf{legal} if $a_1>0$, the other $a_i\geq 0$, and one of the following holds.\vspace{1mm}
    \begin{itemize}
        \item \emph{Condition 1:} We have $m<L$ and $a_i = c_i$ for $1 \leq i \leq m$,
        \item \emph{Condition 2:} There exists $s\in \{1,\dots,L\}$ such that
        \[
            a_1\ =\ c_1, \ \ a_2\ =\ c_2, \ \ \dots, \ \ a_{s-1}\ =\ c_{s-1}, \ \ a_s\ <\ c_s,
        \]
        $a_{s+1}, \dots , a_{s+\ell} = 0$ for some $\ell \geq 0$, and $\{a_{s+\ell+i}\}_{i=1}^{m-s-\ell}$is legal.
    \end{itemize}
Additionally, we let the empty decomposition be legal for $N=0$.
\end{definition}

\begin{remark}
Informally, a legal decomposition is one where we cannot use the recurrence relation to replace a linear combination of summands with another summand, and the coefficient of each summand is appropriately bounded; other authors \cite{DG, Ste} use the phrase $G$-ary decomposition for a legal decomposition. For example, if $H_{n+1} = 3H_n + 2H_{n-1} + 4H_{n-2}$, then $H_5 + 3H_4 + 2H_3 + 3H_2$ is legal, while $H_5 + 3H_4 + 2H_3 + 4H_2$ is not since we can replace $3H_4 + 2H_3 + 4H_2$ with $H_5$; similarly $6H_5+2H_4$ is not legal because the coefficient of $H_5$ is too large.
\end{remark}

We now state an important generalization of Zeckendorf's theorem, then describe what recurrences and sequences we are studying, followed by our results. See \cite{BBGILMT, BM, BCCSW, CFHMN, CFHMNPX, DFFHMPP, Ho,MNPX, MW, Ke, Len} for more on generalized Zeckendorf decompositions, and \cite{GT, MW} for a proof of Theorem \ref{thm:genzeckthmforPLRS}.
\begin{theorem}[Generalized Zeckendorf's theorem for a PLRS]\label{thm:genzeckthmforPLRS}
Let $\{H_n\}_{n=1}^{\infty}$ be a \emph{Positive Linear Recurrence Sequence}. Then
\begin{enumerate}
    \item there is a unique legal decomposition for each non-negative integer $N \geq 0$, and
    \item there is a bijection between the set $\mathcal{S}_n$ of integers in $[H_n,H_{n+1})$ and the set $\mathcal{D}_n$ of legal decompositions $\sum_{i\,=\,1}^n a_i\, H_{n+1-i}$.
\end{enumerate}
\end{theorem}

While this result is powerful and generalizes Zeckendorf's theorem to a large class of recurrence sequences, it is restrictive in that the leading term must have a positive coefficient. We examine what happens in general to existence and uniqueness of legal decompositions if $c_1=0$. Some generalizations were studied in \cite{CFHMN, CFHMNPX} on sequences called the $(s,b)$-Generacci sequences. In-depth analysis was done on the $(1,2)$-Generacci sequence, later called the Kentucky sequence,  and the Fibonacci Quilt sequence; the first has uniqueness of decomposition while the second does not.

\newpage

\begin{definition}\label{def:zlrrdefinition}
    We say a recurrence relation is an $s$-deep \textbf{Zero Linear Recurrence Relation (ZLRR)} if the following properties hold.
    \begin{enumerate}
        \item \emph{Recurrence relation:} There are non-negative integers $s, L, c_1, \dots, c_L$ such that
        \begin{equation}\label{eqn:zlrsrecurrence}
            G_{n+1}\ =\ c_1\, G_n + \cdots + c_s\, G_{n+1-s} + c_{s+1}\, G_{n-s}+ \cdots + c_L\, G_{n+1-L},
        \end{equation}
        with $c_1,\dots, c_s = 0$ and $L, c_{s+1}, c_L$ positive.
        \item \emph{No degenerate sequences:} Let $S =\{m  \mid c_m \neq 0\}$ be the set of indices of positive coefficients. Then $\gcd(S) = 1$.
    \end{enumerate}
\end{definition}

\begin{remark}
We impose the second restriction to eliminate recurrences with undesirable properties, such as $G_{n+1} = G_{n-1} + G_{n-3}$, where the odd- and even-indexed terms do not interact. Any sequence satisfying this recurrence splits into two separate, independent subsequences. Also note that $0$-deep ZLRR's are just PLRR's, whose sequences and decomposition properties are well-understood.
\end{remark}

\begin{definition}\label{def:sdeepzlrs}
    We say a sequence $\{G_n\}_{n\,=\,1}^{\infty}$ of positive integers arising from an $s$-deep ZLRR is an $s$-deep \textbf{Zero Linear Recurrence Sequence (ZLRS)} if $G_1 = 1$, $G_2 = 2$, \dots , $G_{s+1} = s+1$ and for $s+2 \leq n \leq L$,

    \begin{equation}\label{eqn:initcondzlrs}
        G_n = \begin{cases}n, & c_{s+1} \leq s, \\ c_{s+1}\,G_{n-s-1} + c_{s+2}\,G_{n-s-2} + \cdots + c_{n-1}\, G_1+1, & c_{s+1} > s. \end{cases}
    \end{equation}

    We call a decomposition $N=\sum_{i=1}^m a_i H_{m+1-i}$ of a positive integer and its associated sequence $\{a_i\}_{i=1}^m$ \textbf{legal}, if $a_i \geq 0$, and one of the following conditions hold:\vspace{1mm}
    \begin{itemize}
        \item \emph{Condition 1:} We have $a_1 = 1$ and $a_i = 0$ for $2 \leq i \leq m$.
        \item \emph{Condition 2:} We have $s<m<L$ and $a_i = c_i$ for $1\leq i\leq m$.
        \item \emph{Condition 3:} There exists $t \in \{s+1,\dots,L\}$ such that
        \[
            a_1\ =\ c_1,\ \ a_2\ =\ c_2,\ \ \dots,\ \ a_{t-1}\ =\ c_{t-1},\ \ a_t\ <\ c_t,
        \]
        $a_{t+1}, \dots, a_{t+\ell}=0$ for some $\ell \geq 0$, and $\{a_{s+\ell+i}\}_{i\,=\,1}^{m-t-\ell}$ is legal.
    \end{itemize}
Additionally, we let the empty decomposition be legal for $N=0$.
\end{definition}

The idea behind Condition 1 is if $N$ appears in the sequence, say $N= G_n$, then we allow this to be a legal decomposition. This is necessary for there to be a legal decomposition for $N=1$ for any $s$-deep ZLRS. 

\begin{remark}\label{rem:lagonacciexception}
    We note one special case for the initial conditions. If $Z_{n+1} = Z_{n-1} + Z_{n-2}$ (a recurrence relation we call the ``Lagonaccis'' as it has a similar recurrence relation to the Fibonaccis, but the terms ``lag'' behind and grow slowly), then $Z_1 = 1$, $Z_2 = 2$, $Z_3 = 4$, $Z_4 = 3$, $Z_5 = 6$, and so on.\footnote{We use $Z_n$ because the Lagonacci's are easy to study, with  interesting properties, usually requiring special attention. For an example of more standard behavior, consider $Y_{n+1} = 2Y_{n-1} + 2Y_{n-2}$, with $Y_1=1$, $Y_2 = 2$, $Y_3=3$, $Y_4 = 6$, \dots.}
\end{remark}

Similar to the initial conditions of a PLRS, we construct the initial conditions in such a way to guarantee existence of legal decompositions. The main idea behind the definition of legal decompositions is if $N$ does not appear in the sequence (i.e., $N \neq G_n$ for any $n\in\mathbb{N}_0$), then for some $m\in\mathbb{N}_0$, $G_m \leq N < G_{m+1}$,\footnote{Note that if $4 \leq N < 3$, then $N$ is not an integer, so we reach no contradiction with the special initial condition case.} and we \textbf{cannot} use $G_m, G_{m-1},\dots,G_{m-s+1}$ in the decomposition of $N$. Let us illustrate this with an example.

\begin{example}\label{ex:lagonaccidecomp10}
    Consider again the Lagonacci sequence $Z_{n+1} = Z_{n-1} + Z_{n-2}$, with the first terms
    \[
        1,\ 2,\ 4,\ 3,\ 6,\ 7,\ 9,\ 13,\ 16,\ \dots,
    \]
    and let us decompose $N = 10$. Since $Z_7 = 9 \leq 10 < 13 = Z_8$, we \emph{\textbf{cannot}} use $Z_7=9$ in its decomposition. So, we use the next largest number, $Z_6=7$, and get $10 = 7+3 = Z_6 + Z_4$. This is a legal 1-deep ZLRS decomposition. However, notice that we can also have $10 = 6+4 = Z_5 + Z_3$.
\end{example}

The above example suggests the following questions. \emph{Is uniqueness of decomposition lost for all ZLRSes? If so, is it lost for finitely many numbers? For infinitely many numbers? For all numbers from some point onward?}\\

We offer two approaches to addressing these questions. This paper focuses on generalizing Zeckendorf's theorem to $s$-deep ZLRSes, while \cite{MMMMS2} converts $s$-deep ZLRR's to PLRR's with the \textit{Zeroing Algorithm}.

\subsection{Main Results}\label{sec:mainresults}

This paper presents the first approach which is summarized immediately below in Theorems \ref{thm:generalzeckZLRS} and \ref{thm:lossofuniqueness1ZLRS} along with Conjectures \ref{conj:nonunique} and \ref{conj:unique}. The proofs of the theorems are presented in Section \ref{sec:zlrslegal}. Section \ref{sec:conclusion} concludes the paper with
open questions for future research.

\begin{theorem}[Generalized Zeckendorf's theorem for $s$-deep ZLRSes]\label{thm:generalzeckZLRS}
    Let $\{G_n\}_{n\,=\,1}^{\infty}$ be an $s$-deep \emph{Zero Linear Recurrence Sequence.} Then there exists a legal decomposition for each non-negative integer $N\geq 0$.
\end{theorem}

\begin{theorem}[Loss of Uniqueness of Decomposition for a family of $s$-deep ZLRSes]\label{thm:lossofuniqueness1ZLRS}
    Let $\{G_n\}_{n\,=\,1}^{\infty}$ be an $s$-deep \emph{Zero Linear Recurrence Sequence} such that $c_{s+1}>s$, $c_{s+2}>0$ and $c_L>1$. Then uniqueness of decomposition is lost for at least one positive integer $N$.
\end{theorem}

We also have some conjectures relating to uniqueness of decompositions for $s$-deep ZLRSes, based on empirical evidence.

\begin{conjecture}\label{conj:nonunique}
  Let $\{G_n\}_{n\,=\,1}^{\infty}$ be an $s$-deep \emph{Zero Linear Recurrence Sequence} such that $c_{s+1}>s$. Then uniqueness of decomposition is lost for at least one positive integer $N$.
\end{conjecture}

\begin{conjecture}\label{conj:unique}
  Let $\{G_n\}_{n\,=\,1}^{\infty}$ be an $s$-deep \emph{Zero Linear Recurrence Sequence} with recurrence relation
  \[
    G_{n+1} = G_{n-2}+c\,G_{n-3},
  \] with $c \geq 4$. Then there exists a unique decomposition for each positive integer $N$.
\end{conjecture}

The proof for Theorem \ref{thm:generalzeckZLRS} is a mostly straightforward strong induction proof. The difficulty arises with the initial conditions, which are split into two cases. Theorem \ref{thm:lossofuniqueness1ZLRS} is proved by finding an integer with at least two decompositions for this family of $s$-deep ZLRSes.

\section{ZLRS-Legal Decompositions}\label{sec:zlrslegal}

We prove Theorems \ref{thm:generalzeckZLRS} and \ref{thm:lossofuniqueness1ZLRS} in Sections \ref{sec:existence} and \ref{sec:uniqueness}, respectively.

\subsection{Existence}\label{sec:existence}

Given an $s$-deep ZLRS, we use a strong inductive argument to show that a greedy-type algorithm always terminates in a legal decomposition. However, we need to make sure the decomposition is legal. Therefore, at each step, we use the largest coefficient possible, depending on the coefficients of the given $s$-deep ZLRS, and making sure we do not have more terms than is legal. We first illustrate this with an example.

\begin{example}\label{ex:decomposition example}
Consider the 1-deep ZLRS $G_{n+1} = 2\,G_{n-1}+2\,G_{n-2}$, which has initial conditions $G_1=1, G_2=2, G_3=3$. The first few terms of this sequence are \[G_4 = 6,\  G_5 = 10,\  G_6= 18,\  G_7=32,\  G_8 = 56,\  G_9 = 100,\  G_{10}=176.\] Let us decompose $N=164$ using the greedy algorithm. Since $G_9 = 100 < 164 < 176=G_{10}$, and $s=1$, we must use $G_8 = 56$ in the decomposition. Since $c_1 = c_2 = 2$, we can use $G_8$ a maximum of two times, and $G_7$ a maximum of one time, which gives us a total of $m=2\,G_8 + G_7 = 2*56+32 = 144$.\\

We must now decompose $N-m = 164-144=20$ with the remaining terms, $G_1,\dots,G_6$. Since $G_6=18 < 20 < 32=G_7$, we can repeat the same process as before, and ``add" the decomposition of 20, which is $20 = 2*10 = 2\,G_5$, into the decomposition of 164. We get \[164 = 2*56 + 32 + 2*10 = 2\,G_8+G_7+2\,G_5.\] Notice that $G_7$ can not be legally used in the decomposition of 20, because the legal decomposition requirements only allow $G_7$ to be used at most once. It is therefore absolutely necessary for $20 < G_7$.
\end{example}

\begin{proof}[Proof of Theorem \ref{thm:generalzeckZLRS}]
By Definition \ref{def:zlrrdefinition} the $s$-deep ZLRS has the form of \eqref{eqn:zlrsrecurrence}. We first prove that the greedy algorithm terminates in a legal decomposition for all integers $N$ up to and including the last initial condition. For a base case for the induction argument we must consider the following 3 cases. Note that case 3 only applies to a specific sequence. For the base cases we use the maximum number of $G_t$s possible consistent with legality.\\

Case 1: If $c_{s+1} \leq s$, then the initial conditions are the first $L$ integers. So, by Condition (1), we trivially have a legal decomposition for all of the initial conditions. \\ \

Case 2: If $c_{s+1} > s$, then the initial conditions are specially constructed so that we guarantee existence of legal decompositions. We do so by adding the smallest integer that cannot be legally decomposed by the previous terms. We illustrate this with an example.
\begin{example}
Let us take 1-deep ZLRSes of the form $G_{n+1} = c_1\, G_{n-1} + c_2 \,G_{n-2}$, where $c_1 > 1$ and $c_2 > 0$. The initial conditions start with $G_1 = 1$ and $G_2 = 2$. Assuming $G_3 > G_2$, we know all $N$ with $G_2 < N < G_3$ cannot use $G_2 = 2$ in their decomposition, so we can only use $G_1=1$. We also have a restriction of only being able to use $G_1=1$ at most $c_1$ times. So, the first number we cannot legally decompose is $c_1+1$, thus, $G_3 = c_1+1$, which comes by construction as well. By a similar argument, $G_4 = 2c_1 + c_2 + 1$.
\end{example}

Case 3 (Special): If the ZLRS is the Lagonaccis, then we must consider the first four terms in the sequence instead of the first three terms. However, since all four integers appear in the sequence ($Z_1=1, Z_2=2, Z_3=4,$ and $Z_4=3$), we still get a trivial legal decomposition for the first four positive integers.\\
\

For the inductive step of the proof we assume that all integers up to and including $N-1$ have a legal decomposition. We must show that $N$ must also have a legal decomposition. Let $G_t \leq N < G_{t+1}$. Two cases must be considered.\\

Case 1: Suppose $N = G_{t}$. Then, trivially, we have a legal decomposition. \\ \

Case 2: Suppose $N > G_{t}$ and let $m \leq N$ be the largest integer decomposed using a legal decomposition involving only summands drawn from $G_t,G_{t-1}, \dots, G_{t-L}$. Suppose $m = a_1G_{t-s-1}$, with $a_1 < c_{s+1}$. To complete the proof, we need to show that $N - m$ can be expressed with the remaining terms. If you recall Example \ref{ex:decomposition example}, it was important that $20<G_7$ so that we could decompose it with the remaining terms. Generalizing from that example, we need $N-m < G_{t-s-1}$. Suppose that does not hold, so instead we have $N-m \geq G_{t-s-1}$. However, this implies that we have not used the maximum number of $G_{t-s-1}$'s in the greedy decomposition, which is a contradiction. So, we now have that $N-m<G_{t-s-1}$. By the strong inductive hypothesis, there exists a legal decomposition of $N-m$. We then add $m$ to this legal decomposition  to obtain the decomposition of $N$. Since the decomposition for $N-m$ is legal, adding $m$ keeps the decomposition legal, by Condition (3) of Definition \ref{def:sdeepzlrs}. Thus we have a legal decomposition for $N$.

\ \ \ \ \ Let $c_i$ be the next non-zero constant in the recurrence relation. We then let $m = c_{s+1}G_{t-s-1} + a_iG_{t-s-i}$ with $a_i < c_{s+i}$. We want to show that $N - m$ can be expressed with the remaining terms. To do so, we need $N-m < G_{t-s-i}$. Suppose not. Then $N-m \geq G_{t-s-i}$. However, this implies that we have not used the maximum number of $G_{t-s-i}$'s in our greedy decomposition, which is a contradiction. So, we have that $N-m<G_{t-s-i}$. By the same reasoning as the previous case, we have a legal decomposition for $N$.

\ \ \ \ \ We continue this argument, taking the next non-zero constant and adding that on to $m$, until we reach this final case.

\ \ \ \ \ Let $m = c_1\,G_{t} + c_2 \,G_{t-1} + \cdots + c_{L-1} \,G_{t+2-L} + (c_L - 1)\,G_{t+1-L}$. This is the largest possible value $m$ can attain with an allowable legal decomposition. We want to show that $N-m < G_{t-L+1}$. Noting $N < G_{t+1}$, we see that
\begin{align}
    N - m &\ = \  N - \pn{c_1\,G_{t} + \cdots + c_{L-1} G_{t+2-L} + (c_L - 1)\,G_{t+1-L}}\nonumber\\
    &\ <\  G_{t+1} -  \pn{c_1\,G_{t} + \cdots + c_{L-1} G_{t+2-L} + (c_L - 1)\,G_{t+1-L}}\nonumber\\
    &\ = \  \pn{c_1\,G_{t} + \cdots + c_{L-1}\, G_{t+2-L} + c_L\, G_{t+1-L}} \nonumber \\
    &\hspace{3cm}-  \pn{c_1\,G_{t} + \cdots + c_{L-1}\, G_{t+2-L} + (c_L - 1)\,G_{t+1-L}}\nonumber\\
    &\ = \  G_{t+1-L}.
\end{align}

\ \ \ \ \ Thus $N-m < G_{t+1-L}$, and in every case we attain a legal decomposition for $N$, as desired. Therefore, by strong induction, we attain a legal decomposition for any positive integer $N$ given a fixed $s$-deep ZLRS.
\end{proof}

\subsection{Loss of Uniqueness}\label{sec:uniqueness}
\begin{proof}[Proof of Theorem \ref{thm:lossofuniqueness1ZLRS}]
We need to show there exists an $N$ such that
\[
    N = c_{s+1}\,G_{L+2}+c_{s+2}\,G_{L+1} + \cdots + c_{L-1}\,G_{s+4}+G_{s+3}+x,
\]
satisfies $G_{L+s+2}<N<G_{L+s+3}$ and $G_{s+3}+x<\min\{2G_{s+3},G_{s+4}\}$. Note that the second condition implies $G_{s+3}<G_{s+3}+x<G_{s+4}$. If such an $N$ exists, it would have two legal decompositions. Namely, $c_{s+1}\,G_{L+2}+c_{s+2}\,G_{L+1} + \cdots + c_{L-1}\,G_{s+4}+G_{s+3}+(x)$ and $c_{s+1}\,G_{L+2}+c_{s+2}\,G_{L+1} + \cdots + c_{L-1}\,G_{s+4}+(G_{s+3}+x)$, where $(n)$ represents the legal decomposition of $n$. If suffices to show that such an $N$ exists. To complete the proof we need the following lemma, whose proof is given after the proof of Theorem \ref{thm:lossofuniqueness1ZLRS}.

\begin{lemma}\label{lemma: lemma of lossofuniquenesstheorem}
Let $\{G_n\}_{n\,=\,1}^{\infty}$ be an $s$-deep ZLRS with recurrence relation
\[
c_1\, G_n + \cdots + c_s\, G_{n+1-s} + c_{s+1}\, G_{n-s}+ \cdots + c_L\, G_{n+1-L},
\]
such that $c_{s+1}>s$, $c_{s+2}>0$ and $c_L>1$. Then
\[
G_{s+L+2}-c_{s+1}\,G_{L+2}-c_{s+2}\,G_{L+1} - \cdots - c_{L-1}\,G_{s+4}-G_{s+3}\ <\ 0.
\]
\end{lemma}

We can now continue with the proof. We first show there exists an $x>0$ such that
\[
  G_{L+s+2}-c_{s+1}\,G_{L+2}-c_{s+2}\,G_{L+1} - \cdots - c_{L-1}\,G_{s+4}-G_{s+3}<x<\min\{G_{s+3},G_{s+4}-G_{s+3}\}.
\]
However, by Lemma \ref{lemma: lemma of lossofuniquenesstheorem}, we see that
\[
G_{L+s+2}-c_{s+1}\,G_{L+2}-c_{s+2}\,G_{L+1} - \cdots - c_{L-1}\,G_{s+4}-G_{s+3}\ \leq\ 0.
\]
So, we require $0 < x < \min\{G_{s+3},G_{s+4}-G_{s+3}\}$. Since $\min\{G_{s+3},G_{s+4}-G_{s+3}\}>1$, such an $x$ exists.

Since $c_L>1$, the condition $G_{s+3}+x<\min\{2G_{s+3},G_{s+4}\}<2G_{s+3}$ implies
\begin{align*}
    N &\ =\ c_{s+1}\,G_{L+2}+c_{s+2}\,G_{L+1} + \cdots + c_{L-1}\,G_{s+4}+G_{s+3}+x\\
    &\ <\ c_{s+1}\,G_{L+2}+c_{s+2}\,G_{L+1} + \cdots + c_{L-1}\,G_{s+4}+2G_{s+3}\\
    &\ \leq\ c_{s+1}\,G_{L+2}+c_{s+2}\,G_{L+1} + \cdots + c_{L-1}\,G_{s+4}+c_L\,G_{s+3}\\
    &\ =\ G_{L+s+3},
\end{align*}
implying such an $N$ exists. Thus, uniqueness of decompositions is not guaranteed for this family of $s$-deep ZLRSes.
\end{proof}

We conclude this section with a proof of Lemma \ref{lemma: lemma of lossofuniquenesstheorem}.

\begin{proof}[Proof of Lemma \ref{lemma: lemma of lossofuniquenesstheorem}]
There are three cases to consider.\\

Case 1: Suppose $L = s+2$. We need to show
\[
    G_{2s+4}-c_{s+1}\,G_{s+4}-G_{s+3} \ <  \ 0.
\]
Using the recursive definition on $G_{2s+4}$, we obtain
\[
    G_{2s+4}-c_{s+1}\,G_{s+4}-G_{s+3}\ =\ c_{s+1}\,G_{s+3}+c_{s+2}\,G_{s+2}-c_{s+1}\,G_{s+4}-G_{s+3}.
\]
Again, using the recursive definition of \eqref{eqn:zlrsrecurrence} and simplifying, we obtain
\[
   c_{s+1}\,G_{s+3}+c_{s+2}\,G_{s+2}-c_{s+1}\,G_{s+4}-G_{s+3}\ =\ (2-G_3)c_{s+1}^2-2c_{s+1},
\]
which is indeed negative, as $G_3>2$ for any $s$-deep ZLRS with $c_{s+1}>s$.\\

Case 2: Suppose $L = s+3$. We need to show
\[
G_{2s+5}-c_{s+1}\,G_{s+5}-c_{s+2}\,G_{s+4}-G_{s+3} \  <  \ 0.
\]
This case follows similar to Case 1 using the recursive definitions of the terms and simplifying, yielding
\[
    (3-H_4)c_{s+1}^2-2c_{s+1}c_{s+2}-c_{s+2}^2-2c_{s+1}-1+(1-c_{s+2})c_{s+3},
\]
which is negative since $c_{s+2}>0$ and $c_{s+1}>s$ implying  $H_4 \ge 4$.\\

Case 3: Suppose $L=s+m$ with $m\geq 4$. We need to show
\begin{equation}\label{eq:lossofuniquenessCase3}
G_{2s+m+2} - c_{s+1}\,G_{s+m+2}-c_{s+2}\,G_{s+m+1}-\cdots -c_{s+m-1}\,G_{s+4}-G_{s+3}\ <  \ 0.
\end{equation}
Using the recursive definition satisfied by the terms, we find \eqref{eq:lossofuniquenessCase3} is equivalent to
\begin{align*}
    &c_{s+1}\,G_{s+m+1} + c_{s+2}\,G_{s+m}+\cdots + c_{s+m-1}\,G_{s+3}+c_{s+m}\,G_{s+2}\\&-c_{s+1}\,G_{s+m+2}-c_{s+2}\,G_{s+m+1}-c_{s+3}\,G_{s+m}-\cdots-c_{s+m-1}\,G_{s+4}-G_{s+3}\\
    &= c_{s+1}\pn{\sum_{i=1}^m c_{s+i}G_{m+1-i}} + c_{s+2}\pn{1+\sum_{i=1}^{m-1}c_{s+i}G_{m-i}} + \cdots+ c_{s+m-x+2}\pn{1+\sum_{i=1}^{x-1}c_{s+i}G_{x-i}} \\
    &+\cdots+c_{s+m-1}\pn{1+\sum_{i=1}^2 c_{s+i}G_{3-i}} + c_{s+m}\pn{1+\sum_{i=1}^1 c_{s+i}G_{2-i}} - c_{s+1}\pn{\sum_{i=1}^m c_{s+i}G_{m+2-i}}\\
    &-c_{s+2}\pn{\sum_{i=1}^m c_{s+i}G_{m+1-i}} - c_{s+3}\pn{1+\sum_{i=1}^{m-1} c_{s+i}G_{m-i}} - \cdots -c_{s+m-y+3}\pn{1+\sum_{i=1}^{y-1}c_{s+i}G_{y-i}}\\
    &-\cdots-c_{s+m-1}\pn{1+\sum_{i=1}^3c_{s+i}G_{4-i}}-\pn{1+\sum_{i=1}^2c_{s+i}G_{3-i}},
\end{align*}
where $2\leq x \leq m$ and $4\leq y \leq m$. To show \eqref{eq:lossofuniquenessCase3} is negative, we consider the coefficients of the terms separately in sub-cases and prove each are non-positive.\\

Sub-case 1: Let us consider the coefficients of $c_{s+i}c_{s+j}$ for $1\leq i,j< m$. Assuming, without loss of generality, that $i\leq j$, we see that for all $m-i+2\leq j < m$, the coefficient of $c_{s+i}c_{s+j}$ is non-positive. If $j < m-i+2$, then we see that the coefficient of $c_{s+i}c_{s+j}$ is $G_{m+2-i-j}-G_{m+3-i-j}$. A simple inductive argument shows that, for $c_{s+1}\geq s$, $G_{n+1}\geq G_n$, implying $G_{m+2-i-j}-G_{m+3-i-j}<0$.\\

Sub-case 2: We now consider the case when each $c_{s+i}$, $1\le i < m$, is not multiplied by other coefficients. For $i=1$, simple examination shows there is no positive part to this coefficient, so it is non-positive. When $i=2$, the positive part is 1, but the negative part is $G_2=2>1$, so this coefficient is negative. For $3\leq i<m$, both the positive and negative part are 1, so the coefficient is zero. Thus, all coefficients of $c_{s+i}$ are non-positive for this range of $i$.\\

Sub-case 3: We finally consider the case where the coefficient of $c_{s+m}$ is non-positive. We see that this coefficient is
\[
    2G_1c_{s+1}-G_2c_{s+1}-G_1c_{s+2}+1\ =\ 1-c_{s+2},
\]
which is non-positive as $c_{s+2}>0$.
\end{proof}

\section{Conclusion and Future work}\label{sec:conclusion}

We have introduced a method to analyze decompositions arising from ZLRSes, raising many fruitful natural questions for future work.

\begin{itemize}
    \item Are Conjectures \ref{conj:nonunique} and \ref{conj:unique} true?\\ \

    \item What is required for an $s$-deep ZLRS to have unique decompositions? If Conjecture \ref{conj:nonunique} is true, we minimally need $c_{s+1}\leq s$. However there are families, such as the Lag-onaccis, with non-unique decompositions and $c_{s+1}\leq s$. What else is needed?  \\ \

    \item Can a stronger result regarding loss of uniqueness be obtained? At the moment, we have constructed some counterexamples. Do the number of decompositions grow exponentially in comparison to the terms of  sequence?
\end{itemize}


MSC2010: 11B39

\ \\

\end{document}